\documentclass[]{article}
\usepackage{tabu}
\usepackage[english]{babel}
\usepackage[T1]{fontenc}
\usepackage[utf8]{inputenc}
\usepackage{amsthm,amsfonts,amsmath,amssymb}
\usepackage{hyperref}
\usepackage{enumerate}
\usepackage[draft]{todonotes}
\usepackage{xcolor}

\usepackage{cite,slashed,mathtools}
\usepackage{amsrefs}

\newcommand{\comment}[1]

\newtheorem{df}{Definition}[section] 
\newtheorem{pro}{Proposition}[section]
\newtheorem{lm}{Lemma}[section]

\newtheorem{teo}{Theorem}[section]
\newtheorem{rem}{Remark}[section]
\newtheorem{cor}{Corollary}[section]
\newtheorem{ex}{Example}[section]

\newtheorem*{pr-no}{Proof}

\newcommand{\Dt}{\ensuremath{\bar{d}}}

\title{Properties of the Digital Root and its Extension to Rational Numbers - an Algebraic Approach}
\date{\today}
\author{
	Lucas T. Cardoso,\thanks{E-mail: \texttt{lucas.cardoso@ufsm.br}}\quad
	Glauber Quadros\thanks{
		E-mail: \texttt{glauber.quadros@ufsm.br}
	}\\[6pt]
	{\small Coordenadoria Acadêmica,
		Universidade Federal de Santa Maria,}\\
	{\small ZIP--CODE Cachoeira do Sul, RS, Brasil}\\[12pt]
}

\begin{document}
	
\providecommand{\sectionautorefname}{section}
\providecommand{\chapterautorefname}{chapter}
\renewcommand*{\sectionautorefname}{Section}
\renewcommand*{\subsectionautorefname}{Subsection}
\renewcommand*{\chapterautorefname}{Chapter}
	
\maketitle

\tableofcontents

\begin{abstract}
	This paper contains an algebraic constructive and self-contained account of the invariance rule of the digital root under division for an arbitrary natural basis representation.
	Both the cases of repeating and non-repeating fractionals are treated. In the preliminary section some known results such as the uniqueness in the representation of a fraction
	are discussed for both the finite and infinite bases cases. Simple examples are introduced throughout the text for illustrative purposes.
\end{abstract}

\section{Introduction}

It is a known fact that behind the easily comprehensible statements in Arithmetic very commonly lies a staggeringly complex solution, when it exists.  Questions such as Fermat's Last Theorem, Goldbach's conjecture, the twin prime conjecture, the odd perfect number conjecture and so many others can be easily understood by a layman. Nonetheless,  for the aforementioned conjectures the last three remain open to this day, the first mentioned being the only one that has been entirely solved by Andrew Wiles \cite{WilesFermat}. The mere study of these and other profound statements has produced new mathematical tools and entire subfields have emerged in Number Theory. Although seemingly simple, Arithmetic is complex and rich enough not to escape from Gödel's Incompleteness Theorem, which haunts Mathematicians to this very day. To quote the prince of Mathematics, K. F. Gauss, " Mathematics is the queen of the sciences and Arithmetic is the queen of Mathematics".

From a very early age we are educated mathematically through the decimal numeral system. This system is so firmly rooted within us that it may be already passing through genetically to our descendants. Although very practical, since we have ten fingers to count, the decimal numeral system is no more special than any other numeral system based on a different basis. In any case, there are some curious and interesting properties associated with the digit $9$, which are frequently used by magicians to fool their audiences into guessing which number one has picked by asking one to perform certain calculations with the chosen number. An amazing example that could let people stumped is by asking them to find the digit $A$ such that $2A99561=[3(523+A)]^2$. At first it seems an impossible feature, unless the person notices that the number on the right side of the equation is divisible by 9, therefore the number on the left is also divisible by $9$ and so is the sum of digits of the corresponding number, that is $2+A+9+9+5+6+1=32+A$ and therefore, $A=4$. This and many other curiosities regarding the digit $9$ are humorously presented in Beiler's book \cite{Beiler}. Essentially, the same role played by the digit $9$ in the decimal basis representation is played by the digit $k-1$ in the representation to the basis $k$. 

This paper was designed to be the more self-contained it could be in order to provide the reader with a smooth ride through the entire text, but without attempting to reinvent the wheel. With that in mind, \autoref{preliminaries}, which is divided into four subsections, intends to lay down a complete background for the comprehension of the two main theorems, \autoref{main1} and \autoref{main2}. Although most of the results in this preliminary section are well-known in the literature, we insisted on presenting them with proofs, remarks and examples, not as a pedantic act, but as a pedagogical artifice for a constructive and progressive understanding of the tools that are going to be explored in the main sections. In addition, still regarding the preliminary section, some of the results are exhibited in a unique fashion, which in and of itself makes it worth reading. Instead of describing some properties using purely calculational methods, as achieved by \cite{Izmirli} and \cite{Costa}, we have explored an algebraic approach to digital roots. The first subsection introduces some basic nomenclatures and definitions. The main purpose of the second subsection, besides presenting some common notions on modular arithmetic, is the partitioning of the ring of integers modulo $n$ by orbits obtained from the left action on it of its multiplicative subgroup, which will play a vital role in the proof of \autoref{main1}. The third subsection develops both the finite and the infinite representation theorems to a fixed base $k\in\mathbb{N}$, where the latter is going to be the starting point in the proof of \autoref{main2}. The fourth and final subsection introduces the notion of digital sum, additive persistence and digital root, which are the key elements for the attaining of our main results. \autoref{semdizima} is centered around  \autoref{main1}, which essentially establishes behaviour patterns regarding the digital root function to the base $k$, and its correlation to orbits, defined in \autoref{mod-aret}, in the ring of integers modulo $k$. \autoref{comdizima} treats the case of a repeating fractionals to the basis $k$, with similar patterns emerging from the repetend. The paper is finalized with a conclusion section.

\section{Preliminaries}\label{preliminaries}

This section presents most of the tools that are going to be used throughout this paper. As mentioned in the introduction section, a good part of the results presented here is treated in  classic books on Algebra and Number Theory such as \cite{Andrews}, \cite{Jacobson} and \cite{GarciaLequain}. However, we do not advise the reader to skip it. 

\subsection{Basic notions}\label{basicnotions}

Throughout this entire paper we shall use the following terminologies $\mathbb{N}=\{0,1,\cdots\}$, $\{0,1,\cdots,n\}=:\langle n\rangle$, and $\mathbb{N}_p:=\mathbb{N}\setminus\{0,1,\cdots,p-1\}=\mathbb{N}\setminus \langle p-1\rangle$. Given an arbitrary set $X$, we conveniently define $X^+:=\{x\in X: x\geq0\}$. Moreover, the symbols $\lfloor\cdot\rfloor$ and $\lceil\cdot\rceil$ represent the floor and the ceiling functions, respectively, where $\lfloor\cdot\rfloor:D
\subseteq\mathbb{R}\rightarrow\mathbb{Z}$ with $\lfloor x\rfloor=\max\{m\in\mathbb{Z}: m\leq x\}$, and $\lceil\cdot\rceil:D
\subseteq\mathbb{R}\rightarrow\mathbb{Z}$ with $\lceil x\rceil=\min\{m\in\mathbb{Z}: m\geq x\}$. 
 

 
\color{black}
\begin{df}
	Let $a,b,c\in\mathbb{Z}$ and $c\neq 0$, then we say $a$ is congruent to $b$ modulo $c$, symbolically $a\equiv b \ (\text{mod}\ c)$, if $(a-b)/c\in\mathbb{Z}$, or equivalently, $c|(a-b)$.
\end{df}

\begin{ex}
	Since $(16-4)/3=4\in\mathbb{Z}$, then $16\equiv 4 \ (\text{mod}\ 3)$. On the other hand, $(16-5)/3=11/3\notin\mathbb{Z}$, which implies $16\not\equiv  5 \ (\text{mod}\ 3)$.
\end{ex}

\begin{df}
	Two numbers $a,b\in\mathbb{Z}$ are said to be \textbf{coprime} or \textbf{relatively prime} if $\text{gcd} \ (a,b)=1$. 
\end{df}

\begin{lm}\label{lema}
	Let $a,b,q\in\mathbb{Z}$ such that $ab\equiv 0 \ (mod \ q)$ and $a$ and $q$ are coprime, then $b\equiv 0 \ (mod \ q)$.
\end{lm}  

\begin{proof}
	From the first condition of the hypothesis, $\exists k\in\mathbb{Z}$ such that $ab=kq$, which implies $q \mid ab$. Since from the second hypothesis, $\text{gcd} \ (a,q)=1$, we must have $q \mid b$, that is, $b\equiv 0 \ (mod \ q)$. 
\end{proof}



\begin{lm}\label{kk-1}
	Let $k\in\mathbb{N}_2$, then $k$ and $k-1$ are coprime, that is, $\text{gcd} \ (k,k-1)=1$.
\end{lm} 


\begin{proof}
	Let $k=k_1^{l_1}\cdots k_d^{l_d}$ and $k-1=q_1^{s_1}\cdots q_e^{s_e}=:q$ be the prime factorizations of $k$ and $k-1$, respectively. Suppose $\exists m>1$ such that $\text{gcd} \ (k,k-1)=m$, then $\exists q_{j_0}$, where $1\leq j_0\leq e$ such that $q_{j_0}|m$ and consequently $q_{j_0}|k_1^{l_1}\cdots k_d^{l_d}$. However, since $k_1^{l_1}\cdots k_d^{l_d}=q_1^{s_1}\cdots q_e^{s_e}+1$
	\begin{equation}
	\frac{k_1^{l_1}\cdots k_d^{l_d}}{q_{j_0}}=\left(q_1^{s_1}\cdots q_{j_0}^{s_{j_0}-1} \cdots q_e^{s_e}+\frac{1}{q_{j_0}}\right)\notin\mathbb{N},
	\end{equation}
	which is a contradiction, therefore $\text{gcd} \ (k,k-1)=1$. 
\end{proof}

\begin{rem}\label{gcdrk-1}
	Note that $\text{gcd} \ (k,k-1)=1 \iff \text{gcd} \ (r,k-1)=1$ where $r$ is any divisor of $k$. 
\end{rem}

\begin{df}\label{fractional}
	Let $q\in\mathbb{Q}$, we say $q$ is a \textbf{non-repeating or terminating fractional to the base $k\in\mathbb{N}_2$} if $\exists \rho\in\mathbb{N}$ such that $k^\rho q\in\mathbb{Z}$ and that $q$ is a \textbf{repeating fractional to the base $k$} otherwise. Moreover, we may define  $\rho_0:=\min \{\rho\in\mathbb{N}: k^\rho q\in\mathbb{\mathbb{Z}}\}$ which is called the \textbf{minimum exponent} of $q$ to the base $k$. 
\end{df}

\begin{rem}
	Since repeating and terminating fractionals dichotomize $\mathbb{Q}$, given $k\in\mathbb{N}_2$, we may write $\mathbb{Q}=\mathbb{Q}_{T_k}\dot\cup\mathbb{Q}_{R_k}$, where $\mathbb{Q}_{T_k}$ and $\mathbb{Q}_{R_k}$ represent the set of terminating and repeating fractionals to the base $k$, respectively. Notice also that since for every basis $k\in\mathbb{N}_2$, the minimum exponent for an integer is $\rho_0=0$, then $\forall k\in\mathbb{N}_2, \ \mathbb{Z}\subset\mathbb{Q}_{T_k}$.
\end{rem}

\begin{lm}\label{teofraciff}
	Let $n,m\in\mathbb{N}_1$ be coprime numbers and $k\in\mathbb{N}_2$, then $n/m\in\mathbb{Q}_{T_k}$ if and only if $m=k_1^{l_1}k_2^{l_2}\cdots k_d^{l_d}$, where $l_i\geq 0$ for $1\leq i\leq d$ and $k_1,k_2,\cdots, k_d$ are the $d$ prime divisors of $k$.
\end{lm}

\begin{proof}
	Suppose first that $m=k_1^{l_1}k_2^{l_2}\cdots k_d^{l_d}$, then 
	\begin{equation}
	\frac{1}{m}=\frac{1}{k_1^{l_1}k_2^{l_2}\cdots k_d^{l_d}}
	\end{equation} 
	Since $k_1,k_2,\cdots,k_d$ are the prime divisors of $k$, then by the Fundamental Theorem of Arithmetic $\exists l'_1,l'_2,\cdots, l'_d\in\mathbb{N}$ such that $k=k_1^{l'_1}k_2^{l'_2}\cdots k_d^{l'_d}$. Therefore,
	\begin{equation}
	\frac{k^\rho}{m}=\frac{\left(k_1^{l'_1}k_2^{l'_2}\cdots k_d^{l'_d}\right)^\rho}{k_1^{l_1}k_2^{l_2}\cdots k_d^{l_d}}=k_1^{\rho l'_1-l_1}k_2^{\rho l'_2-l_2}\cdots k_d^{\rho l'_d-l_d}.
	\end{equation}
	Thus, choosing
	\begin{equation}
	\rho\geq \underset{1 \leq i \leq d}{\max}\left\{\left\lceil\frac{l_i}{l'_i}\right\rceil\right\},
	\end{equation}
	we have 
	\begin{equation}
	k^\rho\frac{n}{m}\in\mathbb{N},
	\end{equation}
	and consequently, by definition, $n/m$ is a terminating fractional to the base $k$. Conversely, suppose $n/m\in\mathbb{Q}_{T_k}$, then by \autoref{fractional}, $\exists \rho\in\mathbb{N}$ such that $k^\rho(\frac{n}{m})\in\mathbb{N}_1$ and since $\text{gcd} \ (n,m)=1$ and $m$ must divide $k^\rho$, $\exists l_1,l_2,\cdots, l_d\in\mathbb{N}$ such that $m=k_1^{l_1}k_2^{l_2}\cdots k_d^{l_d}$, where $l_i\geq 0$ for $1\leq i\leq d$ and $k_1,k_2,\cdots, k_d$ are the $d$ prime divisors of $k$.
	
\end{proof}

\subsection{Modular Arithmetic}\label{mod-aret}

Since the congruence relation is an equivalence relation on $\mathbb{Z}$, we may define the \textbf{congruence classes modulo $n$} or \textbf{residue classes modulo $n$} as
\begin{equation}
\overline{x}:=\{y\in\mathbb{Z}: \ x\equiv y \ (mod \ n)\}=\{x+jn\}_{j\in\mathbb{Z}}.
\end{equation}
From this definition we can introduce the set of all congruence classes modulo $n$, i.e. $\mathbb{Z}_n:=\{\overline{0},\overline{1},\cdots,\overline{n-1}\}$. The operations of addition and multiplication on $\mathbb{Z}$ induce the following operations on the set $\mathbb{Z}_n$:
\begin{align*}
\oplus: \mathbb{Z}_n\times\mathbb{Z}_n&\longrightarrow\mathbb{Z}_n\\
(\overline{x},\overline{y})&\longmapsto\overline{x+y}\\
\odot: \mathbb{Z}_n\times\mathbb{Z}_n&\longrightarrow\mathbb{Z}_n
\end{align*}
It is easy to check that $(\mathbb{Z}_n,\oplus,\odot)$ is a ring called the \textbf{ring of integers modulo $n$}. Moreover, let
\begin{equation}\label{gama1n}
((\mathbb{Z}_n)^*,\odot):=\{\overline{\delta}\in\mathbb{Z}_n : \text{gcd} \ (\delta,n)=1\}=:\Gamma_1^n.
\end{equation}
The set defined in \eqref{gama1n} is actually the  multiplicative subgroup of $(\mathbb{Z}_n,\oplus,\odot)$. Equivalently, $\Gamma_1^n$ may be defined as the  group of units of the ring of integers modulo n, that is
\begin{equation}
((\mathbb{Z}_n)^*,\odot):=\{\overline{\delta}\in\mathbb{Z}_n : \exists\gamma\in\mathbb{Z}: \delta\gamma\equiv 1 \ (mod \ n)\}
\end{equation}
In order to generalize the definition of $\Gamma_1^n$ to a set where the greatest common divisors are allowed to be different from $1$, let us introduce the following notations. Let $n\in\mathbb{N}$ and $\mathcal{D}_n:=\{\delta\in\mathbb{N}: \delta | n \}=\{\delta_1,\cdots,\delta_{|\mathcal{D}_n|}\}$ be the set of all divisors of $n$.
It is clear that $\mathcal{D}_0=\mathbb{N}_1$ with $|\mathcal{D}_0|=\aleph_0$
and  $1\leq|\mathcal{D}_n|\leq n$ for $n>0$. Furthermore, let $n=n_1^{l_1}\cdots n_m^{l_m}$ be the prime decomposition of $n$, then it is also clear that $|\mathcal{D}_n|=\prod_{i=1}^{m}(l_i+1)$. \par
Let $\delta_i\in\mathcal{D}_n$, then we may define 
\begin{equation}
\Gamma_{\delta_i}^n:=\{\overline{\delta}\in\mathbb{Z}_n : \text{gcd} \ (\delta,n)=\delta_i\}.
\end{equation} 
For the sake of contextual harmony and parsimony we shall use the notation $\Gamma_{1}^n$ instead of $((\mathbb{Z}_n)^*,\odot)$. Notice, that since $1\notin\Gamma_{\alpha}^n$ if $\alpha\neq1$, only $\Gamma_1^n$ is a group. Also, if $\alpha$ is not a divisor of $n$, then $\Gamma_{\alpha}^n=\emptyset$. Furthermore, since the relation defined by $\overline{a}\sim \overline{b} :\Leftrightarrow  \text{gcd} \ (a,n)= \text{gcd} \ (b,n)$ is an equivalence relation on $\mathbb{Z}_n$, it becomes clear that the sets $\Gamma^{n}_{\delta_i}$ partition the set $\mathbb{Z}_n$, that is

\begin{equation}\label{partition}
\coprod_{i=1}^{|\mathcal{D}_n|}\Gamma^{n}_{\delta_i}=\mathbb{Z}_n. 
\end{equation} 

\begin{df}
	Let $(G,\odot)$ be a group and $g\in(G,\odot)$, then we denote its \textbf{cardinality} or \textbf{order} by $|G|$. We say \textbf{$g$ has order $\delta$ in $G$}, denoted by $\text{ord}(g)=\delta$, if $\delta$ is the smallest positive integer for which $g^\delta=e$, where $e$ is the identity element in $G$ and $g^\delta$ represents $g$ $\odot$-multiplied $\delta$ times. Moreover, the order of an element $g$ is equal to the order of its cyclic subgroup $\langle g\rangle:= \{g^k: \  k\in\mathbb{Z}\}$, that is, $\text{ord}(g)=|\langle g\rangle|$.
\end{df}	

\begin{rem}
	It has already been shown that the group $\Gamma_1^n$ is cyclic if and only if $n=1, 2, 4, p^j$ or $2p^j$ where $p$ is an odd prime number \cite{IrelandRosen}. Thus, only in these cases, $\exists g\in\Gamma_1^n$ such that $\{1,g,\cdots, g^{\varphi(n)-1}\}=\Gamma_1^n$, where $\varphi$ is the \textbf{Euler's totient function}, which gives the cardinality of $\Gamma_1^n$, that is, $\varphi(n)=|\Gamma_1^n|$.
\end{rem}

\begin{lm}\label{orderofelement}
	The order of any element in $\mathbb{Z}_n$ divides $n$. In addition, if $\delta$ is a divisor of $n$, then the set of elements of order $\delta$ in $\mathbb{Z}_n$ is non-empty.
\end{lm}

\begin{proof}
	 Let $\bar{x}\in\mathbb{Z}_n$, then since  $\langle\bar{x} \rangle := \{\bar{x}k: \  k\in\mathbb{Z}\}$ is a subgroup of  $\mathbb{Z}_n$, by Lagrange's theorem, $\text{ord}(\bar{x}) | n$. Furthermore, for the second claim, let $\delta$ be a divisor of $n$ and $A_\delta$ denote the set of all elements of order $\delta$ in $\mathbb{Z}_n$, we shall prove that $\overline{\left(\frac{n}{\delta}\right)}\in A_\delta$. First, note that $\overline{\left(\frac{n}{\delta}\right)}\bar{\delta}=\bar{0}$, now suppose $\exists\delta'\in\mathbb{N}$ such that $\delta'|\delta$ and $\overline{\left(\frac{n}{\delta}\right)}\bar{\delta'}=\bar{0}$.  Since $\delta'|n$, $\exists\alpha,\beta\in\mathbb{N}$ such that  $n=\delta'\alpha $ and $\delta=\beta\delta'$,  it follows that $\frac{n}{\delta}\delta'=\frac{n}{\beta}$ and consequently $\overline{\left(\frac{n}{\delta}\right)}\bar{\delta'}=\overline{\left(\frac{n}{\beta}\right)}$, that is $\overline{\left(\frac{n}{\beta}\right)}=\bar{0}$, which implies $\beta=1$ and $\delta=\delta'$. Thus, $\overline{\left(\frac{n}{\delta}\right)}\in A_\delta$.
	
\end{proof}

For future endeavors a particular group action on $\mathbb{Z}_n$ will be very useful:

\begin{pro}\label{G-action}
	The map $\phi: \Gamma^n_1\times\mathbb{Z}_n\rightarrow\mathbb{Z}_n$, given by $\phi(\bar{g},\bar{x}):=\bar{g}\bar{x}$ is a left group action on $\mathbb{Z}_n$, which decomposes $\mathbb{Z}_n$ into $|\mathcal{D}_n|$-many distinct orbits given by $\Gamma_{1}^n\cdot\bar{\delta_i}$, where $\delta_i\in\mathcal{D}_n$.	
\end{pro}

\begin{proof}
	Let $n$ and $\delta|n$ be positive integers. Let $\psi: \Gamma^n_1\rightarrow\Gamma^\delta_1$ be the map given by $\psi(\bar{g})=\tilde{g}$, where $\bar{g}$ and $\tilde{g}$ are the congruence classes of the integer $g$ modulo $n$ and $\delta$, respectively. Then, $\psi$ is a well-defined surjective homomorphism with $|\ker\psi|=\varphi(n)/\varphi(\delta)$. 
	Now, according to \autoref{orderofelement}, it is always possible to pick an element $\bar{x}$ of order $\delta$ in $\mathbb{Z}_n$. Let us consider the stabilizer subgroup of  $\Gamma^n_1$ with respect to $\bar{x}$, i.e. $S_{\bar{x}}:=\{\bar{g}\in\Gamma^n_1: \ \bar{g}\bar{x}=\bar{x}\}$. If $\bar{x}=\bar{0}$, then $\delta=1$ and it trivially follows that $\tilde{g}=\tilde{1}$. Let $\bar{x}\neq\bar{0}$, then from $(\bar{g}-\bar{1})\bar{x}=\bar{0}$ and $\bar{\delta}\bar{x}=\bar{0}$ it follows that $\bar{\delta}=\bar{g}-\bar{1}$, that is, $\tilde{g}=\tilde{1}$. Therefore, we may rewrite $S_{\bar{x}}$ as $S_{\bar{x}}:=\{\bar{g}\in\Gamma^n_1: \ \tilde{g}=\tilde{1}\}$, which corresponds to the kernel of the homomorphism $\psi$. Therefore, $|S_{\bar{x}}|=\varphi(n)/\varphi(\delta)$. By the orbit-stabilizer theorem, we have $|\mathbb{Z}_n|=|\Gamma_{1}^n\cdot\bar{x}||S_{\bar{x}}|$ and consequently $|\Gamma_{1}^n\cdot\bar{x}|=\varphi(\delta)$. Now, let $\bar{y}\in\Gamma_{1}^n\cdot\bar{x}$, then $\exists\bar{g}'\in\Gamma_{1}^n: \ \bar{y}=\bar{g}'\bar{x}$ which implies $\bar{\delta}\bar{y}=\bar{g}'\bar{\delta}\bar{x}=\bar{0}$. Moreover, suppose $\exists q\in\mathbb{N}: \ q|\delta$ and $\bar{q}\bar{y}=\bar{0}$, then $\bar{q}\bar{y}=\bar{g}'\bar{q}\bar{x}=0\implies \ \bar{q}\bar{x}=\bar{0}$ and since $\bar{x}$ is, by assumption, of order $\delta$, it follows that $q=\delta$ and every element of $\Gamma_{1}^n\cdot\bar{x}$ is of order $\delta$. Therefore, since the orbit $\Gamma_{1}^n\cdot\bar{x}$ has $\varphi(\delta)$ elements of order $\delta$, we may choose another representative for the orbit of $\bar{x}$, for instance let $\delta_i:=\overline{\left(\frac{n}{\delta}\right)} | n$, then $\Gamma_{1}^n\cdot\bar{x}=\Gamma_{1}^n\cdot\bar{\delta_i}$. Finally, since $\sum_{\delta_i\in\mathcal{D}_n}\varphi(\delta_i)=n$, we have
	\begin{equation}\label{partition2}
	\mathbb{Z}_n=\coprod_{i=1}^{|\mathcal{D}_n|}\Gamma_{1}^n\cdot\bar{\delta_i}
	\end{equation}
\end{proof}

\begin{teo}\label{orbiteq}
	Let $\delta_i$ be a divisor of $n$ and $\Gamma_1^n\cdot \overline{\delta_i}= \{\overline{g}\cdot\overline{\delta_i}, \ \overline{g}\in\Gamma_{1}^n \} $ be its orbit under the action of $\Gamma_1^n$. Then
	\begin{equation}
	\Gamma_1^n\cdot \overline{\delta_i}=\Gamma_{\delta_i}^n
	\end{equation}
\end{teo}
\begin{proof}
	Let us first prove that $\Gamma_1^n\cdot \overline{\delta_i}\subseteq \Gamma_{\delta_i}^n$. Let $\overline{\delta}\in\Gamma_1^n\cdot \overline{\delta_i}$, then $\exists\overline{g}\in\Gamma_1^n$ such that $\overline{\delta}=\overline{g}\overline{\delta_i}$ which implies $\delta =g\delta_i+nl$ for some $l\in\mathbb{Z}$. Therefore, $\text{gcd} \ (\delta,n)=\text{gcd} \ (g\delta_i+nl,n)=\text{gcd} \ (g\delta_i,n)=\text{gcd} \ (\delta_i,n)=\delta_i$.Thus, by definition, $\overline{\delta}\in\Gamma_{\delta_i}^n$.\\
	Now let us prove the inverse inclusion, that is,  $\Gamma_{\delta_i}^n\subseteq\Gamma_1^n\cdot \overline{\delta_i} $. By equations \eqref{partition} and \eqref{partition2}, we have 

	\begin{equation}\label{partitionequality}
	\coprod_{i=1}^{|\mathcal{D}_n|}\Gamma^{n}_{\delta_i}=\coprod_{i=1}^{|\mathcal{D}_n|}\Gamma_{1}^n\cdot\bar{\delta_i}.
	\end{equation}

	Equation \eqref{partitionequality} together with the inclusion established in the first part of this proof imply $\Gamma_1^n\cdot \overline{\delta_i}=\Gamma_{\delta_i}^n$.
	 
\end{proof}

\color{black}


\begin{ex}\label{Z9}
	Let us consider the case of $n=9$, which will develop an important role in the proofs of our main theorems on the invariance of the digital root. In this case, we have the set $\mathbb{Z}_9=\{\overline{0},\overline{1},\cdots,\overline{8}\}$. The divisors of $9$ are $1$, $3$ and $9$, and consequently  $\mathcal{D}_9=\{1,3,9\}$. Moreover, using \autoref{G-action} and \autoref{orbiteq}, we have the following $3$ distinct orbits
	\begin{align*}
	&\Gamma_{1}^9:=\{\overline{\delta}\in\mathbb{Z}/9\mathbb{Z} : \text{gcd} \ (\delta,9)=1\}=\{\overline{1},\overline{2},\overline{4},\overline{5},\overline{7},\overline{8}\}\\
	&\Gamma_{3}^9:=\{\overline{\delta}\in\mathbb{Z}/9\mathbb{Z} : \text{gcd} \ (\delta,9)=3\}=\{\overline{3},\overline{6}\}\\
	&\Gamma_{9}^9:=\{\overline{\delta}\in\mathbb{Z}/9\mathbb{Z} : \text{gcd} \ (\delta,9)=9\}=\{\overline{0}\}.
	\end{align*}

\end{ex}


\begin{pro}\label{gcdr^jk-1}
Let $k\in\mathbb{N}_2$. If	$r\in\mathcal{D}_k$, then $\bar{r}\in\Gamma_1^{k-1}$.
\end{pro}

\begin{proof}
	Let $r\in\mathcal{D}_k$, that is, $r$ is a divisor of $k$, then as stated in \autoref{gcdrk-1},  $\text{gcd} \ (r,k-1)=1$ which implies $\bar{r}\in\Gamma_1^{k-1}$.
\end{proof}

\subsection{Finite and infinite representation theorems}\label{fin and inf rep}

\begin{teo}[Finite Basis Representation Theorem]\label{FBRT}
	Let $q\in\mathbb{Q}^+$, then $q\in\mathbb{Q}^+_{T_k}$ if, and only if. there exists a unique finite representation for $q$ with respect to the base $k$ as
	
	\begin{equation}\label{basek}
	q=\sum_{j=0}^{p_i}\alpha_{p_i-j} k^{p_i-\rho_0-j}=\alpha_{p_i}k^{p_i-\rho_0}+\alpha_{p_i-1}k^{p_i-\rho_0-1}+\cdots+\alpha_{0}k^{-\rho_0},
	\end{equation}
	where $\alpha_{p_i}, \alpha_{p_i-1},\cdots, \alpha_{0}\in\langle k-1\rangle$,  such that  $0 < \alpha_{p_i}<k$ if $p_i-\rho_0> 0$ and $0 < \alpha_{0}<k$ if $\rho_0>0$.
\end{teo}



\begin{proof}
	Let $q\in\mathbb{Q}^+_{T_k}$, then $k^{\rho_0}q\in\mathbb{N}$ where $\rho_0$ is the minimum exponent of $q$ with respect to the base $k$. By the Finite Basis Representation Theorem for natural numbers, $k^{\rho_0}q$ can be written as
	\begin{equation}\label{equation t_0}
	k^{\rho_0}q=\sum_{j=0}^{p_i}\alpha_{p_i-j} k^{p_i-j}=\alpha_{p_i}k^{p_i}+\alpha_{p_i-1}k^{p_i-1}+\cdots+\alpha_{0},
	\end{equation}
	where $\alpha_0,\alpha_1,\cdots \alpha_{p_i}\in\langle k-1\rangle$ and $\alpha_{p_i}\neq 0$. From equation \eqref{equation t_0} we obtain
	\begin{equation}
	q=\sum_{j=0}^{p_i}\alpha_{p_i-j} k^{p_i-\rho_0-j}=\alpha_{p_i}k^{p_i-\rho_0}+\alpha_{p_i-1}k^{p_i-\rho_0-1}+\cdots+\alpha_{0}k^{-\rho_0}.
	\end{equation}
	The uniqueness of the representation with respect to a specific basis follows directly from the Finite Basis Representation for natural numbers (see \cite{Andrews}, pages 8-10).
	The reverse implication is straightforward, since if $q\in\mathbb{Q}^+$ is given by equation \eqref{basek}, then $k^{\rho_0}q\in\mathbb{N}$, which by \autoref{fractional} means $q\in\mathbb{Q}^+_{T_k}$.
\end{proof}

\color{black}

\begin{rem}
	According to the preceding theorem, a natural number $n$ may be uniquely represented in terms of its digits with respect to the base $k$.
	 Therefore, we shall conveniently depict a natural number $n$ with respect to the base $k$ as a juxtaposition of its digits in the following fashion $[\alpha_{p_i}\alpha_{p_i-1}\cdots \alpha_{0}]_k$, and in the case of the base $10$ we write simply $[\alpha_{p_i}\alpha_{p_i-1}\cdots \alpha_{0}]_{10}:=\alpha_{p_i}\alpha_{p_i-1}\allowbreak\cdots \alpha_{0}$. For non-integer numbers with finite representation given by \eqref{basek} we may write $[\alpha_{p_i}\cdots\alpha_{\rho_0},\alpha_{\rho_0-1}\cdots \alpha_{0}]_k$. 
	 By using this convention, the finite basis representation theorem for natural numbers may be extended for non-repeating fractionals as was done in \ref{FBRT}.
	
\end{rem}

\begin{ex}
	Let us see some illustrative examples. 
	\begin{enumerate}[(i)]
		\item $7205=7\cdot 10^3+2\cdot 10^2+0\cdot 10^1+5\cdot 10^0$
		\item $[425]_6=4\cdot 6^2+2\cdot 6^1+5\cdot 6^0=144+12+5=161.$
		\item $
		[101011]_2  =1\cdot 2^5+0\cdot 2^4+1\cdot 2^3+0\cdot 2^2+1\cdot 2^1+1\cdot 2^0 $\\
		  $= (1\cdot 3^3+0\cdot 3^2+1\cdot 3^1+2\cdot 3^0)+(2\cdot 3^1+2\cdot 3^0)+(2\cdot 3^0)+(1\cdot 3^0)$\\
		 $= 1\cdot 3^3+1\cdot 3^2+2\cdot 3^1+1\cdot 3^0= [1121]_3$
		 \item $	72.05=7\cdot 10^1+2\cdot 10^0+0\cdot 10^{-1}+5\cdot 10^{-2}.$
		 \item $[4.25]_6=4\cdot 6^0+2\cdot 6^{-1}+5\cdot 6^{-2}=\dfrac{161}{36}.$
		 \item Let us consider the hexadecimal basis with $A=10$, $B=11$, $C=12$, $D=13$, $E=14$ and $F=15$. Then
		 \begin{equation*}
		 [2A7E]_{16}=2\cdot 16^3+A\cdot 16^2+7\cdot 16^1+E\cdot 16^0=10878.
		 \end{equation*}
	\end{enumerate}
\end{ex}
\begin{rem}
	Notice in the last item of the previous example that the number $161/36$ is, according to \autoref{teofraciff}, a repeating fractional number to the base $10$ (i.e. a repeating decimal number), since $\text{gcd} \ (161,36)=1$ and $36=2^23^3$, where $3$ is not a divisor of $10$.
	Notwithstanding, $161/36$ is a terminating fractional to the base $6$ and therefore admits a finite basis representation to this base.
\end{rem}

\begin{lm}\label{k^m-geq}
	Let $k\in\mathbb{N}_2$, $a_j\in\langle k-1\rangle$ for all $j\in\mathbb{N}$, then
	\begin{equation}\label{k^m-ineq}
	k^m\geq \sum_{j=0}^{+\infty}a_{j} k^{m-1-j},
	\end{equation}
	and equality holds if and only if $a_j=k-1 \ \forall j\in\mathbb{N}$. 
\end{lm}

\begin{proof}
	The series in \eqref{k^m-ineq} is clearly convergent. Let $a_j=k-1, \forall j\in\mathbb{N}$, then
	\begin{equation}
	\sum_{j=0}^{+\infty}(k-1) k^{m-1-j}=(k^m-k^{m-1})+(k^{m-1}-k^{m-1})+\cdots=k^m.
	\end{equation}
If $\exists j_0\in\mathbb{N}: \ a_{j_0}\neq k-1$, then clearly
\begin{equation}
k^m=\sum_{j=0}^{+\infty}(k-1) k^{m-1-j}>\sum_{j=0}^{+\infty}a_{j} k^{m-1-j}.
\end{equation} 	
	
\end{proof}

\begin{teo}[Infinite Basis Representation Theorem]\label{IBRT}
	Let $k\in\mathbb{N}_2$. Then, $\forall x\in\mathbb{R}^+$, there exists a unique infinite representation for $x$ with respect to the base $k$ as
	\begin{equation}\label{basekex}
	x=\sum_{j=0}^{+\infty}a_{p_i-j} k^{p_i-j},
	\end{equation}
	where $p_i\in\mathbb{Z}$ and $a_{p_i-j}\in\mathbb{N}$ such that $0\leq a_{p_i-j}< k$ for all $j\geq 0$, and $\forall N>0, \ \exists j_0 > N: a_{p_i - j_0} \neq 0$. Furthermore, defining
	\begin{equation}
	x_n:=\sum_{j=0}^{n}a_{p_i-j} k^{p_i-j},
	\end{equation}
	and $\{x_n\}_{n\in\mathbb{N}_0}=:S\subset\mathbb{R}$, $x$ is a limit point of $S$.
\end{teo}



\begin{proof}
	Let $x\in\mathbb{R}^+$ and $k$ be a positive integer greater than $1$, then $\exists m\in\mathbb{Z}$ such that 
	$k^{m}<x$. Let $m_0$ be the greatest integer with such a property, that is, $m_0:=\max\{m\in\mathbb{Z}: k^m<x\}$ and $0\leq a_{m_0}<k$ be the greatest value for which it still holds $a_{m_0}k^{m_0}<x$. Once again, let us choose $0\leq a_{m_0-1}<k$ as the greatest value for which $a_{m_0}k^{m_0}+a_{m_0-1}k^{m_0-1}<x$. 
	Since $\mathbb{R}$ is densely ordered, the indefinite repetition of this process is valid and we might construct a non-decreasing sequence $\{x_n\}_{n\in\mathbb{N}}$ such that its general term $x_n$  is given by  
	\begin{equation}
	x_n:=\sum_{j=1}^{n}a_{m_0-j+1} k^{m_0-j+1}<x.
	\end{equation}
	
	Since $\{x_n\}_{n\in\mathbb{N}}$ is a non-decreasing sequence that is bounded from above by $x$, by the Monotone Convergence Theorem the sequence converges to its supremum. It still remains to show that $x=\sup_{n\in\mathbb{N}}\{x_n\}$. In this regard, let $\epsilon>0$, then, since $\mathbb{N}$ is unbounded from above, $\exists N_0\in\mathbb{N}$ such that
	\begin{equation}
	\epsilon>k^{m_0-N_0+1}.
	\end{equation}
	Thus,
	\begin{equation}\label{SN0+}
	x_{N_0}+\epsilon>x_{N_0}+k^{m_0-N_0+1}=a_{m_0}k^{m_0}+\cdots+(a_{m_0-N_0+1}+1)k^{m_0-N_0+1},
	\end{equation}
	However, by the very definition of the sequence $\{x_n\}_{n\in\mathbb{N}}$, $a_{m_0-N_0+1}$ is the greatest coefficient of $k^{m_0-N_0+1}$ for which $x_{N_0}<x$. Therefore, from equation \eqref{SN0+} we have $x_{N_0}+\epsilon > x$, that is, $x_{N_0} > x-\epsilon$. 
	Since $\epsilon$ is an arbitrary positive real number, it follows that $x=\sup_{n\in\mathbb{N}}\{x_n\}$. 
	Moreover, by construction $x_{N'} < x$ $\forall N' \in \mathbb{N}$ which means $x_{N'}\in (x-\epsilon,x+\epsilon)\setminus\{x\}$ $\forall N' \geq N_0$ which proves $x$ is a limit point of $\{x_n\}_{n\in\mathbb{N}}$.
	
	For the uniqueness of the representation, let 
		\begin{equation}\label{basekex'}
		x=\sum_{j=0}^{+\infty}b_{p'_i-j} k^{p'_i-j},
		\end{equation}
	be an infinite representation for $x$, with $p'_i\in\mathbb{Z}$ and $b_{p'_i-j}\in\mathbb{N}$ such that $0\leq b_{p'_i-j}< k$ for all $j\geq 0$, and $\forall N>0, \ \exists j_1 > N: b_{p'_i - j_1} \neq 0$. Suppose, without any loss of generality, that $p'_i>p_i$, then $\exists l\in\langle k-1 \rangle$ such that $p'_i=p_i+l$. Thus, subtracting \eqref{basekex} from \eqref{basekex'} we obtain
	\begin{align}
	0&=\sum_{j=0}^{+\infty}b_{p_i+l-j} k^{p_i+l-j}-\sum_{j=0}^{+\infty}a_{p_i-j} k^{p_i-j}\nonumber\\
	&=\sum_{j=0}^{l-1}b_{p_i+l-j} k^{p_i+l-j}+\sum_{j=0}^{+\infty}(b_{p_i-j}-a_{p_i-j}) k^{p_i-j}\label{x-x}.
	\end{align}	
	If we had $p'_i=p_i$, then equation \eqref{x-x} would still be valid with the omission of the first term on the right-hand side. Equation \eqref{x-x} may be rewritten as
	\begin{equation}
	b_{p_i+l}k^{p_i+l}+\cdots+b_{p_i+1}k^{p_i+1}=\sum_{j=0}^{+\infty}(a_{p_i-j}-b_{p_i-j}) k^{p_i-j}.
	\end{equation}
	From \autoref{k^m-geq}, we have
	\begin{equation}
	k^{p_i+l}>k^{p_i+l-1}>\cdots>k^{p_i+1}=\sum_{j=0}^{+\infty}(k-1) k^{p_i-j},
	\end{equation}
	Since the coefficients $b_{p_i+l-j}$ are all non-negative and $-(k-1)\leq a_{p_i-j}-b_{p_i-j}\leq k-1$,
	we must have $b_{p_i+l-j}=0$ for $0\leq j\leq l-2$ and
	\begin{equation}\label{bk}
	b_{p_i+1}k^{p_i+1}=\sum_{j=0}^{+\infty}(a_{p_i-j}-b_{p_i-j}) k^{p_i-j},
	\end{equation}
	where, either $b_{p_i+1}=1$, or $b_{p_i+1}=0$. In the first case, according to \autoref{k^m-geq}, $a_{p_i-j}-b_{p_i-j}=k-1, \ \forall j\in\mathbb{N}$, which is equivalent to the assertion that $a_{p_i-j}=k-1$ and $b_{p_i-j}=0$ for all $j\in\mathbb{N}$, which in turn contradicts the hypothesis that $\forall N>0, \ \exists j_1 > N: b_{p'_i - j_1} \neq 0$. Therefore, we must have $b_{p_i+1}=0$. We want to prove that  $a_{p_i-j}=b_{p_i-j}, \ \forall j\in\mathbb{N}$ and we shall do so by induction on $j$. Let us prove the initial case, that is, for $j=0$. Suppose $b_{p_i}>a_{p_i}$, then equation \eqref{bk} can be rewritten as
	\begin{equation}\label{ind-j=0}
(b_{p_i}-a_{p_i})k^{p_i}=\sum_{j=1}^{+\infty}(a_{p_i-j}-b_{p_i-j}) k^{p_i-j},
	\end{equation}
	From \autoref{k^m-geq}, for equation \eqref{ind-j=0} to hold, either $(b_{p_i}-a_{p_i})=1$, or $(b_{p_i}-a_{p_i})=0$. In the first case, it follows that
	$a_{p_i-j}=k-1$ and $b_{p_i-j}=0$ for all $j\in\mathbb{N}_1$, which again contradicts the hypothesis that $\forall N>0, \ \exists j_1 > N: b_{p'_i - j_1} \neq 0$. The case $a_{p_i}>b_{p_i}$ is completely analogous and also leads to contradiction. Therefore, $b_{p_i}=a_{p_i}$. For the final step of the induction, suppose $a_{p_i-j}=b_{p_i-j}, \ \forall j\in\langle n\rangle $, where $n\in\mathbb{N}$, then from equation \eqref{bk} it follows
	\begin{equation}
	\sum_{j=0}^{+\infty}(a_{p_i-j}-b_{p_i-j}) k^{p_i-j}=\sum_{j=n+1}^{+\infty}(a_{p_i-j}-b_{p_i-j}) k^{p_i-j}=0,
	\end{equation}
	which may be rewritten as
	\begin{equation}\label{ind-j=n}
	(b_{p_i-n-1}-a_{p_i-n-1})k^{p_i-n-1}=\sum_{j=0}^{+\infty}(a_{p_i-n-2-j}-b_{p_i-n-2-j}) k^{p_i-n-2-j}
	\end{equation}
	Again from \autoref{k^m-geq}, for equation \eqref{ind-j=n} to hold, either $(b_{p_i-n-1}-a_{p_i-n-1})=1$, or $(b_{p_i-n-1}-a_{p_i-n-1})=0$. In the first case, it follows that
	$a_{p_i-j}=k-1$ and $b_{p_i-j}=0$ for all $j\in\mathbb{N}_{n+2}$, which again contradicts the hypothesis that $\forall N>0, \ \exists j_1 > N: b_{p'_i - j_1} \neq 0$. Therefore, $b_{p_i-n-1}=a_{p_i-n-1}$, hence $a_{p_i-j}=b_{p_i-j}, \ \forall j\in\langle n+1\rangle $ and the final step of the induction is done.
	
	

\end{proof}

\begin{rem}\label{repreg}

		The Infinite Basis Representation Theorem allows a unique representation in a given basis $k\in\mathbb{N}_2$ for any positive real number $x$ in such a way that there are infinitely many non-zero digits, that is, $q:=[a_{p_i}\cdots a_0,a_{-1}\cdots]_k$. In the case of a positive rational number, however, there is always a finite string of digits that repeats itself indefinitely and so, a finite amount of numbers is enough to describe it. In the case $q\in\mathbb{Q}^+$, we shall use the following notation  
\begin{align}
q&=[a_{p_i}\cdots a_{0},\cdots a_{-\rho_0}\overline{a_{\rho_0-1}\cdots a_{-\rho_0-T}}]_k\\
&=[a_{p_i}\cdots a_{0},\cdots a_{-\rho_0}]_k+[ 0,\cdots \overline{a_{\rho_0-1}\cdots a_{-\rho_0-T}}]_k\\
&=:Reg_k(q)+\overline{Rep_k(q)},
\end{align}		
where the string of digits under the horizontal line is the repeating portion of $q$, called its \textbf{repetend} and denoted by $Rep_k(q)$, $T\in\mathbb{N}$ is the number of digits in the repeating cycle, called the \textbf{period} or \textbf{length of the repetend}, and $Reg_k(q)$ denotes the non-repeating string of digits of $q$, called its \textbf{regular part}. Clearly, there is no possibility of confusion between a finite and an infinite representation, since in the latter $T\geq 1$ and a horizontal line over a finite string of digits is always present. In the former case though, we might assume $T=0$, where $q$ reduces to  $q:=[a_{p_i}\cdots a_{0},\cdots a_{-\rho_0}]_k$, and the digits are now considered with respect to its finite representation to the base $k$. Furthermore, if $q\in\mathbb{Q}_{T_k}$, then $q$ has both a finite and an infinite representation to the base $k$, which are evidently distinct from one another. As an example, $[4.25]_6=[4.24\overline{5}]_6$.
		
\end{rem}

\subsection{Digital sum, additive persistence and digital root}\label{dig root and dig sum}

\begin{df}
	Let $n$ be a natural number to the base $k$. Then, the function $d_k:\mathbb{N}\rightarrow\mathbb{N}$ defined by
	\begin{equation}
	d_k(n)=d_k([a_{p_i}\cdots a_{0}]_k):=\sum_{j=0}^{p_i}a_{p_i-j}
	\end{equation}
is called the \textbf{digital sum} of $n$ with respect to the base $k$.
\end{df}

\begin{ex} \
	\begin{enumerate}[(i)]
		\item
		$d_{10}([7205]_{10})=d_{10}(7205)=7+2+0+5=[14]_{10}=14$
		\item $d_6([425]_6)=4+2+5=11=[15]_6$
		\item $
		d_2([101011]_2)=1+0+1+0+1+1=4=[100]_2$
		\item $d_{16}([2A7E]_{16})=2+A+7+E=[21]_{16}$
	\end{enumerate}
\end{ex}

\begin{pro}
	Let $n$ be a natural number to the base $k$. Then, $d_k(n)\leq n$ and equality holds if and only if $n$ is a single digit number to the base $k$. 
\end{pro}

\begin{proof}
	Using the definition of the function $d_k$ it is straightforward that
	\begin{equation}
	d_k(n)=\sum_{j=0}^{p_i} a_{p_i-j}\leq\sum_{j=0}^{p_i}a_{p_i-j}k^{p_i-j}=n.
	\end{equation}
	and clearly equality holds if and only if $p_i-j=0$ for $0\leq j\leq p_i$, which is equivalent to stating that $p_i=0$. Thus, $n=a_0=\sum_{j=0}^{0}a_{p_i-j}=d_k(n)$. 
\end{proof}

\begin{rem}
	Since $d_k(n)\in\mathbb{N}$ we might calculate its digital sum with respect to the base $k$, that is, $d_k([d_k(n)]_k)\leq d_k(n)\leq n$ and this process might be iterated as many times as we please by composing the digital sum function with itself multiple times. For convenience we shall define the $N$-times composition of $d_k$ with itself as
	\begin{equation}
	 d_k^{(N)}(n):=\underset{N \text{-times}}{(d_k\circ\cdots\circ d_k})(n).
	\end{equation}
	Moreover, it is clear that from the definition of $d_k$ and from the fact that $\mathbb{N}$ is bounded from bellow, given a basis $k$, $\forall n\in\mathbb{N},\exists N\in\mathbb{N}:\ d_k^{(N)}(n)=b_0$, where $b_0$ is a single digit number to the base $k$, that is $0\leq b_0\leq k-1$.
\end{rem}

\begin{df}
Let $n\in\mathbb{N}$.Then, the function $\mathcal{A}_k:\mathbb{N}\rightarrow\mathbb{N}$ such that $\mathcal{A}_k(n):=\min\{N\in\mathbb{N}: 0\leq d_k^{(N)}(n)\leq k-1\}$ is called the \textbf{additive persistence} of $n$ with respect to the base $k$.
\end{df}
\begin{rem}
	The additive persistence function gives the minimum number of times we need to add the digits of a particular number recursively until we obtain a single digit number.
\end{rem}	

\begin{ex} \
	\begin{enumerate}[(i)]
		\item
		$d_{10}(7205)=14$ and $ d_{10}(14)=5$ $\implies$ $ \mathcal{A}_{10}(7205)=2$;
		\item $d_6([425]_6)=[15]_6$, $d_6([15]_6)=6=[10]_6$, and $d_6([10]_6)=[1]_6$ $\implies$ $\mathcal{A}_6([425]_6)=3$;
		\item $
		d_2([101011]_2)=[100]_2$ and $d_2([100]_2)=[1]_2$ $\implies$ $\mathcal{A}_2([101011]_2)=2$.
		\item $d_{16}([2A7E]_{16})=2+A+7+E=[21]_{16}$ and $d_{16}([21]_{16})=[3]_{16}$ $\implies$ $\mathcal{A}_{16}([2A7E]_{16})=2$.
	\end{enumerate}
\end{ex}

\begin{df}\label{root1}
The function $r_k: \mathbb{N}\rightarrow \langle k-1\rangle$ defined by $r_k(n):=d_k^{\mathcal{A}_k(n)}(n)$ is called the \textbf{digital root} of $n$ to the base $k$.
\end{df}	
\begin{ex}From  the previous example it follows that $r_{10}(7205)=5$, $r_6([425]_6)=[1]_6$ and $r_2([101011]_2)=[1]_2$.
	
\end{ex}
Some of the previous results may be easily extended from natural to non-repeating rational numbers. 
\begin{df}\label{rootsum2}
	Let $\Dt_k:\mathbb{Q}^+_{T_k}\rightarrow\mathbb{N}$ and $\bar{r}_k:\mathbb{Q}^+_{T_k}\rightarrow\langle k-1 \rangle$ be defined as $\Dt_k(q):=d_k(k^{\rho_0}q)$ and $\bar{r}_k(q):=r_k(k^{\rho_0}q)$ where $\rho_0$ is the minimum exponent of $q$ to the base $k$, then $\Dt_k$ and $\bar{r}_k$ are called the \textbf{terminating fractional digital sum} (TFDS) and the \textbf{terminating fractional digital root} (TFDR) of $q\in\mathbb{Q}^+_{T_k}$ to the base $k\in\mathbb{N}_2$.
\end{df}

\begin{ex} 
	\begin{enumerate}[(i)]
		\item
		$\Dt_{10}([72.05]_{10})=d_{10}(10^2\cdot72.05)=d_{10}(7205)=14$ and analogously $\bar{r}_{10}([72.05]_{10})=r_{10}(10^2\cdot72.05)=r_{10}(7205)=5$
		\item $
		\Dt_2([1.01011]_2)=d_2(2^5[1.01011]_2)=d_2([101011]_2)=[100]_2$ and analogously $\bar{r}_2([1.01011]_2)=r_2(2^5[1.01011]_2)=r_2([101011]_2)=[1]_2$.
	\end{enumerate}
\end{ex}

\section{Terminating Fractionals and Digital Root}\label{semdizima}

%

\begin{lm}\label{lemadkrk}
Let $q\in\mathbb{Q}^+_{T_k}$, then $\Dt_k(q)\equiv \bar{r}_k(q) \ (\text{mod}\ k-1)$.
\end{lm}


\begin{proof}
	If $q=0$, then evidently $\Dt_k(q)=\bar{r}_k(q)=0$,and in particular, $\Dt_k(q)\equiv \bar{r}_k(q) \ (\text{mod}\ k-1)$. Thus, let $q\in\mathbb{Q}^+_{T_k}\setminus\{0\}$ with $\Dt_k(q)\equiv \gamma \ (\text{mod}\ k-1)$ where $0\leq \gamma<k-1$, and $q=[\alpha_{p_i}\cdots\alpha_{p_i-\rho_0},\cdots\alpha_{0}]_k$ be the representation of $q$ to the base $k$, then by definition, 
	\begin{equation}
	\Dt_k(q)=\Dt_k([\alpha_{p_i}\cdots\alpha_{p_i-\rho_0},\cdots\alpha_{0}]_k)=\sum_{j=0}^{p_i}\alpha_{p_i-j}=:q'\in\mathbb{N}. 
	\end{equation}
	Therefore, $\Dt_k(q)\equiv \gamma \ (\text{mod}\ k-1)$ if and only if $q'\equiv \gamma \ (\text{mod}\ k-1)$. Again by the Finite Representation Theorem we may write $q'=[\alpha'_{p'_i}\cdots\alpha'_{0}]_k$ from which it follows that	
	\begin{equation}
	q'':=\Dt_k(q')=\sum_{j=0}^{p'_i}\alpha'_{p'_i-j}\equiv\sum_{j=0}^{p'_i}\alpha'_{p'_i-j}k^{p_i-j} \ (\text{mod}\ k-1)=q'. 
	\end{equation} 
	Thus, $q''\equiv q' \ (\text{mod}\ k-1)\equiv \gamma \ (\text{mod}\ k-1)$, that is $\Dt_k(\Dt_k(q))\equiv \gamma \ (\text{mod}\ k-1)$. By induction it follows that $\forall n\in\mathbb{N}_1, \ \Dt^{(n)}_k(q)\equiv \gamma \ (\text{mod}\ k-1)$ and particularly $\bar{r}_k(q)=\Dt^{\mathcal{A}_k(q)}_k(q)\equiv \gamma \ (\text{mod}\ k-1)$.
\end{proof}



\begin{rem}
	Notice that from the congruence relation $\bar{r}_k(q)\equiv \gamma \ (\text{mod}\ k-1)$ it follows that $\bar{r}_k(q)=\gamma+l(k-1)$, where $l\in\mathbb{Z}$ and since the digital root is a positive single digit number
	
		\[ \bar{r}_k(q) = 
	\begin{cases} 
	\gamma, & \text{if} \quad 0<\gamma<k-1 \\
	k-1, & \text{if} \quad \gamma=0 \ \wedge \ l=1 \\
	0, & \text{if} \quad \gamma=0 \ \wedge \ l=0 \\
	\end{cases}
	.\]
	It is also evident that $\bar{r}_k(q)=0\Leftrightarrow\Dt_k(q)=0\Leftrightarrow  q=0$. From which it follows that $\bar{r}_k(q) =k-1\Leftrightarrow \Dt_k(q)$ is a multiple of $k-1$ and also if $q\in\mathbb{N}$, then it is a multiple of $k-1$.
\end{rem}
	
\begin{teo}\label{main1}
	Let $q\in\mathbb{Q}^+_{T_k}\setminus\{0\}$ and $r\in\mathbb{N}_2$ such that $r$ is a proper divisor of $k$, then $\exists\delta_i\in\mathcal{D}_{k-1}$
	\begin{equation}\label{maineq1}
 \bigg\{\overline{\bar{r}_k\left(\frac{q}{r^j}\right)}\bigg\}_{j\in\mathbb{N}}\subseteq\Gamma^{k-1}_{\delta_i}.
	\end{equation}
\end{teo}

\begin{proof}
		Let us denote $\bar{r}_k(q/r^j)=:R_j$. Since, by \autoref{G-action}, the collection of orbits  $\{\Gamma_{\delta}^{k-1}\}_{\delta\in\mathcal{D}_{k-1}}$ partitions the set $\mathbb{Z}_{k-1}$, given $j\in\mathbb{N}$,  $\exists\delta\in\mathcal{D}_{k-1}$ such that $\overline{\bar{r}_k\left(\frac{q}{r^{j}}\right)}=\overline{R}_j\in\Gamma_{\delta}^{k-1}$.  Let $\Gamma_{\delta_0}^{k-1}$ be the orbit containing $\overline{R}_0=\overline{\bar{r}_k(q)}$. We shall prove that $\forall j\in\mathbb{N}$, $\overline{R}_j\in\Gamma_{\delta_0}^{k-1}$. Suppose, by contradiction, $\exists j_1\neq 0: \ \overline{R}_{j_1}\notin\Gamma_{\delta_0}^{k-1}$, that is, suppose $\exists \delta_1\in\mathcal{D}_{k-1}$ with $\delta_0\neq\delta_1$ such that $\overline{R}_{j_1}\in\Gamma_{\delta_1}^{k-1}$. Denoting $\frac{q}{r^j}=:q_j$, according to \autoref{FBRT}, both $q_{0}=q$ and $q_{j_1}$ may be uniquely represented to the base $k$ as
	
	\begin{equation}\label{q0}
	q_{0}=q=\sum_{l=0}^{p_i}\alpha_{p_i-l} k^{p_i-\rho_0-l},
	\end{equation}
and
	\begin{equation}\label{qj1}
	q_{j_1}=\frac{q}{r^{j_1}}=\sum_{l=0}^{p'_i}\alpha'_{p'_i-l} k^{p'_i-\rho'_0-l},
	\end{equation}
where $\rho_0$ and $\rho'_0$ are the minimum exponents of $q_0$ and $q_{j_1}$ to the base $k$, respectively. Since $\rho_0$ is the minimum exponent of $q$, equations \eqref{q0} and \eqref{qj1} imply 

	\begin{equation}\label{}
	k^{\rho_0}q_{0}=k^{\rho_0}q=\sum_{l=0}^{p_i}\alpha_{p_i-l} k^{p_i-l}\equiv \sum_{l=0}^{p_i}\alpha_{p_i-l}  \ (\text{mod}\ k-1),
	\end{equation}
that is, 
	\begin{equation}\label{q0'}
	k^{\rho_0}q\equiv \Dt_k\left(q_0\right)  \ (\text{mod}\ k-1),
	\end{equation}
and
	\begin{equation}\label{}
	k^{\rho_0}q=r^{j_1}\sum_{l=0}^{p'_i}\alpha'_{p'_i-l} k^{p'_i+\rho_0-\rho'_0-l}\equiv r^{j_1} \sum_{l=0}^{p'_i}\alpha'_{p'_i-l}  \ (\text{mod}\ k-1),
	\end{equation}
that is,
	\begin{equation}\label{qj1'}
	k^{\rho_0}q\equiv r^{j_1}\Dt_k\left(q_{j_1}\right)  \ (\text{mod}\ k-1).
	\end{equation}	
Moreover, by \autoref{lemadkrk},  $\Dt_k\left(q_0\right)\equiv R_0 \ (\text{mod}\ k-1)$ and $\Dt_k\left(q_{j_1}\right)\equiv R_{j_1} \ (\text{mod}\ k-1)$. Thus, from equations \eqref{q0'} and \eqref{qj1'} it follows that 
	\begin{equation}\label{last-eq}
	r^{j_1}R_{j_1}\equiv R_0\ (\text{mod}\ k-1). 
	\end{equation}	
	
Now, since $r\in\mathcal{D}_{k}$, from \autoref{gcdrk-1},  $\forall j\in\mathbb{N}, \text{gcd} \ (r^j,k-1)=1$, from which it follows, particularly, that $\forall j\in\mathbb{N},\overline{r^j}\in\Gamma_1^{k-1}$. Moreover, since the action  $\phi: \Gamma_1^{k-1}\times\mathbb{Z}_{k-1}\rightarrow\mathbb{Z}_{k-1}$, given by $\phi(g,x):=gx$ defined in \autoref{G-action} leaves the orbits invariant,	$\overline{r}^{j_1}\overline{R}_{j_1}\in\Gamma_{\delta_1}^{k-1}$, which is a contradiction since from equation \eqref{last-eq} $\overline{r}^{j_1}\overline{R}_{j_1}=\overline{R}_0\in\Gamma_{\delta_0}^{k-1}$ and the orbits are disjoint. 


\end{proof}

\begin{cor}\label{cor1}
	Let $q\in\mathbb{Q}^+_{T_k}\setminus\{0\}$ and $r\in\mathbb{N}_2$ such that $r$ is a proper divisor of $k$. If $\bar{r}_k(q)\equiv 0 \ (\text{mod}\ k-1)$, then $\bar{r}_k(q/r)\equiv 0 \ (\text{mod}\ k-1)$.
\end{cor}	
	
\begin{proof}
	From \autoref{main1} $\overline{\bar{r}_k(q)}\in\Gamma_{k-1}^{k-1}=\{\overline{0}\}$ and consequently $\overline{\bar{r}_k(q/r)}\in\Gamma_{k-1}^{k-1}=\{\overline{0}\}$, that is $\bar{r}_k(q/r)\equiv 0 \ (\text{mod}\ k-1)$.
\end{proof}	
	
\begin{ex}
	As an example let us look again at the decimal basis case. In this case, as already pointed out in \autoref{Z9}, the orbits are $\Gamma_{1}^{9}=\{\overline{1},\overline{2},\overline{4},\overline{5},\overline{7},\overline{8}\}$, $\Gamma_{3}^{9}=\{\overline{3},\overline{6}\}$, and $\Gamma_{9}^{9}=\{\overline{0}\}$. Since $2$ and $5$ are the proper divisors of $10$, the division of any non-repeating decimal number by powers of $2$ or $5$ will again result in a non-repeating decimal number. Moreover, according to \autoref{main1}, the digital root of two numbers whose ratio are integer powers of $2$ or $5$ lies in the same orbit. As a consequence of \autoref{cor1}, the orbit of $9$ is special in the sense that it contains only one element, namely, $\overline{0}$, which means that if the terminating fractional digital root of a number $q$ is in $\Gamma_{9}^9$, every number whose ratio by $q$ yields integer powers of $2$ or $5$ has digital root equal to $9$.
	This example is interestingly enough, however, it might immediately lead oneself to ask if there is any similar conclusions we can draw on for repeating fractionals such as $1/7$ or $3/11$. The answer is ``yes, we can.'' We cannot apply either definitions \ref{root1} or \ref{rootsum2}, but we could turn our attention to repetends to see if any pattern emerges in this case, which may be investigated by directly looking at the repetend. The following section is devoted to this matter. 
	
\end{ex}	

\begin{ex}
	 For a final example, consider the non-decimal base example given by the sequence of terminating fractionals to the base $8$, $\{[25]_8,$ $[12.4]_8,$ $[5.2]_8,$ $[2.5]_8,$ $[1.24]_8,$$\cdots \}$, which is a geometric progression with respect to the base $8$ with ratio $1/2$. Since the digital root of the number $[25]_8$ is $7$, as a consequence of \autoref{cor1}, dividing $[25]_8$ by powers of $2$ will result in non-fractional numbers to the base $8$ with the same digital root $7$.
\end{ex}

\color{black}

\section{Repeating Fractionals and Digital Root}\label{comdizima}

\begin{teo}\label{main2}
	Let $n\in\mathbb{N}_1$, $s\in\mathbb{N}_2$ such that $n/s$ is an irreducible fraction to the base $k$ and $s=k_1^{l_1}\cdots k_m^{l_m}p$, where $ l_1,\cdots, l_m\in\mathbb{N}$, $k_1,\dots,k_m$ are primes in the prime decomposition of $k$, and $p\in\mathbb{N}_2$ such that $\text{gcd} \ (p,k-1)=\text{gcd} \ (p,k)=1$. Then, $n/s\in\mathbb{Q}_{R_k}$ and letting $n_{s,T}$ be its repetend, $\bar{r}_k(n_{s,T})\equiv 0 \mod(k-1)$.
\end{teo}

\begin{proof}
	The fact that $n/s$ is a repeating fractional is a direct consequence of \autoref{teofraciff}. Thus, by the infinite basis representation theorem, \autoref{IBRT}, we have
	\begin{align}\label{q/r}
	t:=\frac{n}{s}&=\sum_{j=0}^{+\infty}a_{p_i-j} k^{p_i-j},\\
	&=\sum_{j=0}^{p_i+\rho_0}a_{p_i-j}k^{p_i-j}+\sum_{j=p_i+\rho_0+1}^{\infty}a_{p_i-j} k^{p_i-j}\\
	&=\sum_{j=0}^{p_i+\rho_0}a_{p_i-j}k^{p_i-j}+\sum_{j=0}^{\infty}a_{-\rho_0-1-j} k^{-\rho_0-1-j}\\
	&=Reg_k(t)+\overline{Rep_k(t)}\label{cyclic}
	\end{align}
	where $p_i\in\mathbb{Z}$, $a_{p_i-j}\in\langle k-1\rangle$ for all $j\in\mathbb{N}$, and $\rho_0$ is the minimum exponent of the regular part of $q$. According to \autoref{repreg}, the second term in the right-hand side of equation \eqref{cyclic} contains the cyclic part of $t$. In addition, $\exists T\geq 1$ such that $\forall j\in\mathbb{N} \ a_{-\rho_0-1-j}=a_{-\rho_0-1-j-T}$. Furthermore, using the Euclidean algorithm we may rewrite the summation variable $j$ in the cyclic term in \eqref{cyclic} as $j=m+lT$, where $0 \leq m \leq T-1$ and $l\geq 0$, obtaining
	\begin{align}\label{t''gen}
	t&=\sum_{j=0}^{p_i+\rho_0}a_{p_i-j}k^{p_i-j}+k^{-\rho_0-1}\sum_{m=0}^{T-1}\sum_{l=0}^{\infty}a_{-\rho_0-1-m-lT} k^{-(m+lT)}\\
	&=\sum_{j=0}^{p_i+\rho_0}a_{p_i-j}k^{p_i-j}+k^{-\rho_0-1}\sum_{m=0}^{T-1}a_{-\rho_0-1-m}k^{-m}\sum_{l=0}^{\infty} k^{-{lT}}\\
	&=\sum_{j=0}^{p_i+\rho_0}a_{p_i-j}k^{p_i-j}+\frac{k^{T-\rho_0-1}}{k^T-1}\sum_{m=0}^{T-1}a_{-\rho_0-1-m}k^{-m}\label{negpowers}
	\end{align}
	
	Since $t$ is a repeating fractional to the base $k$, $\nexists t\in\mathbb{N}$ such that $k^{r}t\in\mathbb{N}$. However, by letting $t':=k^{\rho_0}t$ we eliminate negative powers of $k$ appearing in the sums of \eqref{negpowers}, that is

	\begin{equation}\label{t''''gen}
	t'=k^{\rho_0}t=\sum_{j=0}^{p_i+\rho_0}a_{p_i-j}k^{p_i+\rho_0-j}+\frac{k^{T-1}}{k^T-1}\sum_{m=0}^{T-1}a_{-\rho_0-1-m}k^{-m},
	\end{equation}	 
	The term $k^T-1$ in the denominator of \eqref{t''''gen} is precisely what makes $t'$ a repeating fractional number to the base $k$. Multiplying the last equation by $k^T-1$ we obtain

	\begin{equation}\label{lasteq}
	(k^T-1)t'=\left[(k^T-1)\sum_{j=0}^{p_i+\rho_0}a_{p_i-j}k^{p_i+\rho_0-j}+\sum_{m=0}^{T-1}a_{-\rho_0-1-m}k^{T-1-m}\right]\in\mathbb{N}
	\end{equation}
	
	From equations \eqref{q/r}, \eqref{t''''gen} and \eqref{lasteq}, we obtain the following result
	
	\begin{equation}\label{t''geninf}
	(k^T-1)k^{\rho_0}\frac{n}{s}=(k-1)\left(\sum_{j=0}^{T-1}k^j\right)k^{\rho_0}\frac{n}{s}=:t''\in\mathbb{N}_1
	\end{equation}
	Since $t''\in\mathbb{N}_1$ and $n/s$ is an irreducible fraction, $s$ must divide the term $k^{\rho_0}(k-1)\sum_{j=0}^{T-1}k^j$. By definition of $k$ and $\rho_0$ we have that $k_1^{l_1}\cdots k_m^{l_m}|k^{\rho_0}$, moreover, since $s=k_1^{l_1}\cdots k_m^{l_m}p$ with $\text{gcd} \ (p,k)=\text{gcd} \ (p,k-1)=1$, it follows that $\left(\sum_{j=0}^{T-1}k^j\right)/p=:u\in\mathbb{N}_1$, that is
	\begin{equation}\label{finalt''}
	(k-1)u\left(\frac{k^{\rho_0}}{k_1^{l_1}\cdots k_m^{l_m}}\right)n=t'',
	\end{equation}
	which implies $t''$ is divisible by $k-1$. In terms of congruences, from expressions \eqref{lasteq}, \eqref{t''geninf}, and \eqref{finalt''} we have
	\begin{equation}
	\left[(k^T-1)\sum_{j=0}^{p_i+\rho_0}a_{p_i-j}k^{p_i+\rho_0-j}+\sum_{m=0}^{T-1}a_{-\rho_0-1-m}k^{T-1-m}\right]\equiv 0\ (mod \ k-1)
	\end{equation}
	and consequently
	\begin{align}
	\sum_{m=0}^{T-1}a_{-\rho_0-1-m}k^{T-1-m}&\equiv \sum_{m=0}^{T-1}a_{-\rho_0-1-m}\equiv 0\ (mod \ k-1)\\
	&=\Dt_k(Rep_k(t))\equiv \bar{r}_k(Rep_k(t))\ (mod \ k-1)
	\end{align}
	
\end{proof}

\begin{rem}
	Since $\forall q\in\mathbb{Q}^+_{T_k}$, $k^{\rho_0}q=n\in\mathbb{N}$, where $\rho_0$ is the minimum exponent of $q$, such that $\bar{r}_k(q)=r_k(k^{\rho_0}q)=r_k(n)$, \autoref{main2} is still valid if we replace $n\in\mathbb{N}_1$, by $q\in\mathbb{Q}^+_{T_k}\setminus\{0\}$. Moreover, \autoref{main2} essentially complements \autoref{main1} when considering divisions resulting in repeating fractionals. For instance, in the $k=10$ case, for any irreducible fraction $n/s$ where $s$ is such that $\exists p\in\mathbb{N}: \ p|s \wedge \forall l_1,l_2,l_3\in\mathbb{N} \ p\neq 2^{l_1}3^{l_2}5^{l_3}$, the number $n/s$ is not only a repeating decimal, but its repetend's digital root is always null, that is, it is a multiple of $9$ irregardless of the value of the numerator's digital root $r_k(n)$ in contraposition to the terminating fractionals case treated in \autoref{main1}.  
\end{rem}

\begin{ex}
	For an example of a sequence of repeating decimals, consider the sequence $\{9/{p_n}\}_{n\in\mathbb{N}}$ where $\{p_n\}_{n\in\mathbb{N}}$ is the sequence of all prime numbers starting with $p_1=7$, that is, $\{9/{p_n}\}_{n\in\mathbb{N}}=\{1.\overline{285714}, 0.\overline{81}, 0.\overline{692307}, 0. \overline{5294117647058823}\allowbreak,\cdots \}$, which explicits the first 4 terms of the sequence, which have repetends whose digital root is $9$.
\end{ex}

\begin{ex}
	For the final example, let us consider the following sequence to the base $8$, $\{[25]_8, [5.1\overline{7}]_8, [1.\overline{1463}]_8,\cdots \}$ which is a geometric progression of ratio $1/5$. Since $[5]_8=5$ is coprime to both $[25]_8=21$ and $[10]_8=8$, then every term, except the first, in the preceding sequence is a repeating fractional to the base $8$ whose repetends' digital root is $7$.
\end{ex}
\color{black}

\section{Conclusion}\label{aaa}

The solution of a mathematical problem is usually strongly dependent upon the problem's characterization within a given mathematical structure. By adequately characterizing a problem we may be able to solve it nicely using theorems and techniques from the structure we have been able to model the problem on. Although the Preliminary section served the purpose of laying the foundations for the resolution of the main theorems, \autoref{main1} and \autoref{main2}, it definitely has its own merit. In \autoref{mod-aret}, \autoref{G-action} establishes  a partition of $\mathbb{Z}_n$ into distinct orbits under the action of the group of units contained in $\mathbb{Z}_n$. Moreover, the number of orbits is shown to be precisely the number of divisors of $n$. \autoref{orbiteq} describes the orbits in a useful and simpler form through the sets $\Gamma_{\delta_i}^n$. In \autoref{fin and inf rep}, we have presented both the finite and infinite representation theorems, \autoref{FBRT} and \autoref{IBRT}, which were useful for the main results in \autoref{main1} and \autoref{main2}, respectively. The last preliminary subsection is devoted for basic notions and definitions regarding both the digital sum and the digital root. In \autoref{semdizima} we focus on terminating fractionals by describing an invariance rule for the digital root in such cases, which is accomplished by \autoref{main1}. Essentially if $q$ and $q/r$ are two terminating fractionals to the basis $k$, then their digital roots lie in the same orbit. As an immediate consequence of this theorem, \autoref{cor1}, it follows that, assuming $r$ and $k-1$ to be coprime, if $q$'s digital root equals to $k-1$, then so does $q/r$'s, since the orbit in this case is the singleton $\Gamma_{k-1}^{k-1}=\{\bar{0}\}$. Within this last corollary lies the decimal basis case, with $k-1=9$ as the ``magic number'', which is humorously depicted in Beiler's book and basically everywhere the digital root is mentioned.  Lastly, \autoref{comdizima} deals with the repeating fractionals case, exhibiting an invariance rule, which, in a sense, is stronger than the terminating fractionals case. Roughly speaking, let $n/r$ be an irreducible repeating fractional to the base $k$ with $r$ coprime to $k-1$, then \autoref{main2} asserts that the digital root of $n/r$'s repetend equals to $k-1$ irregardless of the value of $n$ within the condition pre-established in the theorem's assertion.

\section*{Acknowledgements}
The first author would like to thank José Amâncio dos Santos for his valuable comments on previous drafts of the paper.

\bibliographystyle{abbrv}

\end{document}